\documentclass{sig-alternate-05-2015}
\usepackage{tikz-cd}
\usepackage{picture}
\usepackage[frenchb]{babel}
\usepackage{mathptmx}       
\usepackage{helvet}         
\usepackage{courier}        
\usepackage{type1cm}        
\usepackage{xhfill}                      
\usepackage{xpatch}
\usepackage{textcase}
\makeatletter
\xpatchcmd{\@sect}{\uppercase}{\MakeTextUppercase}{}{}
\xpatchcmd{\@sect}{\uppercase}{\MakeTextUppercase}{}{}
\makeatother
\usepackage{makeidx}         
\usepackage{graphicx}        
\usepackage{multicol}        
\usepackage[bottom]{footmisc}
\usepackage{pgf,tikz}
\usetikzlibrary{arrows}
\usepackage{ulem}
\def\mod{\mathrm{mod}}

\def\ind{\mathrm{ind}}

\def\LCXX{{\cal L}ie_{\mathbb C} \langle\!\langle X \rangle\!\rangle}
\def\LCX{{\cal L}ie_{\mathbb C} \langle X \rangle}

\def\scal#1#2{\langle #1\bv#2 \rangle}
\def\ncp#1#2{#1\langle #2\rangle}
\def\ncs#1#2{#1\langle \!\langle #2\rangle \!\rangle}
\usepackage{amsmath}
\usepackage{amsfonts}
\usepackage{amssymb}

\newcommand{\calA}{{\mathcal A}}

\newcommand{\calH}{{\mathcal H}}
\newcommand{\calC}{{\mathcal C}}

\newcommand{\calJ}{{\mathcal J}}

\newcommand{\N}{{\mathbb N}}
\renewcommand{\P}{\N_{>0}}
\newcommand{\Z}{{\mathbb Z}}
\newcommand{\Q}{{\mathbb Q}}

\newcommand{\C}{{\mathbb C}}

\newcommand{\Frac}[2]{\displaystyle \frac{#1}{#2}}
\newcommand{\Sum}[2]{\displaystyle{\sum_{#1}^{#2}}}

\newcommand{\Lim}[1]{\displaystyle{\lim_{#1}\ }}

\def\pol#1{\langle #1 \rangle}
\def\Lyn{{\mathcal Lyn}}

\def\CX{\C \langle X \rangle}
\def\CY{\C \langle Y \rangle}
\def\abs#1{|#1|}

 \def\shuffle{\mathop{_{^{\sqcup\!\sqcup}}}} 

\gdef\stuffle{\;%
  \setlength{\unitlength}{0.0125cm}%
  \begin{picture}(20,10)(220,580) 
  \thinlines 
  \put(220,592){\line( 0,-1){ 10}} 
  \put(220,582){\line( 1, 0){ 20}} 
  \put(240,582){\line( 0, 1){ 10}} 
  \put(230,592){\line( 0,-1){ 10}} 
  \put(225,587){\line( 1, 0){ 10}} 
  \end{picture}\; 
}

\newtheorem{corollary}{Corollary}
\newtheorem{proposition}{Proposition}
\newtheorem{theorem}{Theorem}
\newtheorem{lemma}{Lemma}
\newtheorem{definition}{Definition}

\newtheorem{example}{Example}

\newtheorem{remark}{Remark}
\newtheorem{diagram}{Diagram}

\newcommand{\Li}{\operatorname{Li}}

\def\L{\mathrm{L}}
\def\H{\mathrm{H}}

\def\P{\mathrm{P}}

\def\deg{\mathrm{deg}}

\newcommand{\poly}[2]{#1 \langle #2 \rangle}

\def\CX{\poly{\C}{X}}

\def\QY{\poly{\Q}{Y}}

\def\CY{\poly{\C}{Y}}

\newcommand{\serie}[2]{#1 \langle \! \langle #2 \rangle \! \rangle}

\def\CXX{\serie{\C}{X}}

\def\QY_0{\Q\left\langle{Y_0}\right\rangle}
\def\pol#1{\langle #1 \rangle}

\def\Lie{{\cal L}ie}

\def\CX{\C \langle X \rangle}
\def\CXX{\serie{\C}{X}}
 
 \newcommand{\calB}{{\cal B}}

\def\Lim{\displaystyle\lim}
\def\Sum{\displaystyle\sum}

\def\Frac{\displaystyle\frac}
\def\path{\rightsquigarrow}

\def\resg{\triangleleft}
\def\resd{\triangleright}
\def\bv{\mid}

\def\abs#1{\bv\!#1\!\bv}

\gdef\minishuffle{{\scriptstyle \shuffle}}

\def\rd{\triangleright}
\def\rg{\triangleleft}
\def\deg{\mathop\mathrm{deg}\nolimits}

\def\binom#1#2{{#1\choose#2}}

\usepackage{standalone}
\usepackage{pgfplots}
\pgfplotsset{compat=newest}
\usepackage{color}
\usepackage{ulem}

\def\scal#1#2{\langle #1\bv#2 \rangle}
\def\ncp#1#2{#1\langle #2\rangle}
\def\ncs#1#2{#1\langle \!\langle #2\rangle \!\rangle}

\begin{document}

\setcopyright{acmcopyright}



\begingroup
\count0=\time \divide\count0by60 
\count2=\count0 \multiply\count2by-60 \advance\count2by\time
\def\2#1{\ifnum#1<10 0\fi\the#1}
\xdef\isodayandtime{
{\2\day-\2\month-\the\year\space\2{\count0}:%
\2{\count2}}}
\endgroup


%

\title{\bf\LARGE{The algebra of Kleene stars of the plane and polylogarithms}}

%
%
%
%

\numberofauthors{3} 
%
\author{
%
%
G\'erard H. E. Duchamp\\
      \affaddr{Universit\'e Paris Nord}\\
       \affaddr{99, av. J-B Cl\'ement}\\
       \affaddr{93430 Villetaneuse, France}\\
      \email{gerard.duchamp@lipn.univ-paris13.fr}
\alignauthor
Hoang Ngoc Minh\\
      \affaddr{Universit\'e de Lille 2}\\
      \affaddr{1 Place D\'eliot}\\
      \affaddr{59000 Lille, France}\\
       \email{hoang@univ.lille2.fr}
\alignauthor
Ngo Quoc Hoan\\
      \affaddr{Universit\'e Paris Nord}\\
      \affaddr{99, av. J-B Cl\'ement}\\
      \affaddr{93430 Villetaneuse, France}\\
       \email{quochoan\_ngo@yahoo.com.vn}}
\maketitle
\begin{abstract}
We extend the definition and study the algebraic properties of the polylogarithm $\Li_T$,
where $T$ is rational series over the alphabet $X=\{x_0, x_1\}$
belonging to $(\ncp{\C}{X}\shuffle\ncs{\C^\mathrm{rat}}{x_0}\shuffle \ncs{\C^\mathrm{rat}}{x_1}, \shuffle, 1_{X^*})$. 
\end{abstract}
%
%
%
%

%
%
\printccsdesc



\section{ Introduction}
In all the sequel of this text,
\begin{enumerate}
\item We consider the differential forms
\begin{eqnarray*}
\omega_0(z)=\frac{dz}z &\mbox{and}& \omega_1(z)=\frac{dz}{1-z}.
\end{eqnarray*}
We denote $\Omega$ the cleft plane ${\mathbb C}-(]-\infty,0]\cup[1,+\infty[)$ 
and $\lambda$ the rational fraction $z(1-z)^{-1}$ belonging to the differential unitary ring
$\calC:={\mathbb C}[z,z^{-1},(1-z)^{-1}]$ with the differential operator $\partial_z:={d}/{dz}$
and with the unitary element
\begin{eqnarray*}
1_{\Omega}:\Omega &\longrightarrow &{\mathbb C},\\
z&\longmapsto &1.
\end{eqnarray*}

\item We construct, over the alphabets
\begin{eqnarray*}
X=\{x_0,x_1\},\quad Y=\{y_k\}_{k\ge1}& \mbox{and}& Y_0=Y\cup\{y_0\}, 
\end{eqnarray*}
totally ordered by $x_0<x_1$ and $y_0>y_1>\cdots$ respectively, 
the  bialgebras
\footnote{Which are all Hopf save the last one due to $y_0$ which is infiltration-like \cite{BDHHT}.} 
\begin{eqnarray*}
(\CX,{\tt conc},\Delta_{\shuffle},1_{X^*},\varepsilon),\\
(\CY,{\tt conc},\Delta_{\stuffle},1_{Y^*},\varepsilon),\\
(\C\langle Y_0\rangle,{\tt conc},\Delta_{\stuffle},1_{Y_0^*},\varepsilon).
\end{eqnarray*}
These algebras, when endowed with their dual laws, are equipped with pure transcendence bases in bijection with the set of Lyndon words
$\Lyn (X),$ $\Lyn (Y)$ and $\Lyn (Y_0)$ respectively. 

Let us consider also the following morphism
\begin{eqnarray*}
\pi^{\circ}_Y:(\C\oplus\CX x_1,{\tt conc})&\longrightarrow&(\CY,.),\\
x_0^{s_1-1}x_1\ldots x_0^{s_r-1}x_1&\longmapsto&y_{s_1}\ldots y_{s_r},
\end{eqnarray*}
for $r\ge1$ and, for any $a\in\C,\pi^{\circ}_Y(a)=a$.
The extension of $\pi^{\circ}_Y$ over $\CX$ is denoted by $\pi_Y :\CX \longrightarrow \CY$
satisfying, for any $p\in\CX x_0$, $\pi_Y(p)=0$. Hence,
\begin{eqnarray*}
\ker(\pi_Y)=\CX x_0&\mbox{and}&\mbox{Im}\;(\pi_Y)=\CY.
\end{eqnarray*}
Let $\pi_X$ be the inverse of $\pi^{\circ}_Y$~:
\begin{eqnarray*}
\pi_X:\CY&\longrightarrow&\C\oplus\CX x_1,\\
y_{s_1}\ldots y_{s_r}&\longmapsto&x_0^{s_1-1}x_1\ldots x_0^{s_r-1}x_1.
\end{eqnarray*}
The projectors $\pi_X$\footnote{With a little abuse of language, $\pi_X$ is now considered as targeted to $\ncp{\C}{X}$.} and $\pi^{\circ}_Y$ are mutual adjoints:
\begin{eqnarray*}
\forall p\in\CX,\forall q\in\CY,&&\scal{\pi_Y(p)}{q}=\scal{p}{\pi_X(q)}.
\end{eqnarray*}
\end{enumerate}

In continuation of \cite{GHMposter,QED}, the principal object of the present work is the {\it polylogarithm} well defined,
for any $r$-uplet $(s_1,\ldots,s_r)\in\C^r$, $r\in\N_+$ and for any $z\in\C$ such that $\abs{z}<1$, as follows
\begin{eqnarray}\label{polylogarithm}
\Li_{s_1,\ldots,s_r}(z):=\Sum_{n_1>\ldots>n_r>0}\frac{z^{n_1}}{n_1^{s_1}\ldots n_r^{s_r}}.
\end{eqnarray}
Then the Taylor expansion of the function $(1-z)^{-1}\Li_{s_1,\ldots,s_r}(z)$ is given by
\begin{eqnarray*}
\Frac{\Li_{s_1,\ldots,s_r}(z)}{1-z}=\Sum_{N\ge0}\H_{s_1,\ldots,s_r}(N)\;z^N,
\end{eqnarray*}
where the coefficient $\H_{s_1,\ldots,s_r}:\N\longrightarrow\Q$ is an arithmetic function,
also called {\it harmonic sum}, which can be expressed as follows
\begin{eqnarray}\label{harmonicsum}
\H_{s_1,\ldots,s_r}(N):=\Sum_{N\ge n_1>\ldots>n_r>0}\Frac1{n_1^{s_1}\ldots n_r^{s_r}}.
\end{eqnarray}

From the analytic continuation of polyzetas \cite{Goncharov,Zhao}, for any $r \geq 1$, if $(s_1,\ldots,s_r)\in\calH_r$ satisfies (\ref{domaine}) then\footnote{For $r \geq 1$, $\zeta(s_1,\ldots,s_r)$ is as a meromorphic function on 
\begin{eqnarray}\label{domaine}
\calH_r=\{(s_1,\ldots,s_r)\in\C^r\vert\forall m =1,\ldots,r,\Re(s_1)+\ldots+\Re(s_m)>1\}.
\end{eqnarray}},
after a theorem by Abel, one obtains the {\it polyzeta} as follows
\begin{eqnarray*}
\Lim_{z \rightarrow 1}\Li_{s_1,\ldots,s_r}(z)=\Lim_{N \rightarrow \infty}\H_{s_1,\ldots,s_r}(N)=\zeta(s_1,\ldots,s_r).
\end{eqnarray*}
This theorem is no more valid in the divergent 
cases (for $(s_1,\ldots,s_r)\in\N^r$)
and require the renormalization of the corresponding divergent polyzetas.
It is already done for the corresponding case of polyzetas at positive 
multi-indices \cite{Daresbury,JSC,cade}
and it is done \cite{FKMT,GZ, MP} and completed in \cite{GHMposter,QED} for the case of  polyzetas at positive multi-indices. 

To study the polylogarithms at negative multi-indices, one relies on \cite{GHMposter,QED} 
\begin{enumerate}
\item the (one-to-one) correspondence between the multi-indices $(s_1,\ldots,s_r)\in\N^r$
and the words $y_{s_1}\ldots y_{s_r}$ defined over  $Y_0$,
\item indexing these polylogarithms by words $y_{s_1}\ldots y_{s_r}$~:
\begin{eqnarray*}
\Li^-_{y_{s_1}\ldots y_{s_r}}(z)=\Li^-_{s_1,\ldots,s_r}(z)=\sum_{n_1>\ldots>n_r>0}{n_1^{s_1}\ldots n_r^{s_r}}\;{z^{n_1}}.
\end{eqnarray*}
\end{enumerate}

In the same way, for polylogarithms at positive indices, one relies on \cite{SLC43,FPSAC97}
\begin{enumerate}
\item the (one-to-one) correspondence between the combinatorial compositions $(s_1,\ldots,s_r)$
and the words $x_0^{s_1-1}x_1\ldots x_0^{s_r-1}x_1$ in $X^*x_1+ 1_{X^*}$
\item the indexing of these polylogarithms by words $x_0^{s_1-1}x_1\ldots x_0^{s_r-1}x_1$~:
\begin{eqnarray*}
\Li_{x_0^{s_1-1}x_1\ldots x_0^{s_r-1}x_1}(z)=\Li_{s_1,\ldots,s_r}(z)=\sum_{n_1>\ldots>n_r>0}\frac{z^{n_1}}{n_1^{s_1}\ldots n_r^{s_r}}.
\end{eqnarray*}
\end{enumerate}
Moreover, one obtained the polylogarithms at positive indices as image by the following isomorphism of the shuffle algebra\footnote{As follows defined on a superset of the of Lyndon words, as pure  
transcendence basis, and extended by algebraic specialization  
\cite{TCS,IMACS}.} \cite{SLC43}
\begin{eqnarray*}
\Li_{\bullet}:(\CX,\shuffle,1_{X^*})&\longrightarrow&(\C\{\Li_w\}_{w\in X^*},\times,1_{\Omega}),\\
x_0^n&\longmapsto&\frac{\log^n(z)}{n!},\\
x_1^n&\longmapsto&\frac{\log^n(1/(1-z))}{n!},\\
x_0^{s_1-1}x_1\ldots x_0^{s_r-1}x_1&\longmapsto&\sum_{n_1>\ldots>n_r>0}\frac{z^{n_1}}{n_1^{s_1}\ldots n_r^{s_r}}.
\end{eqnarray*}
Extending over the set of rational power series\footnote{${\mathbb C}^{\mathrm{rat}}\langle\langle X\rangle\rangle$ is the closure
by rational operations $\{+,{\tt conc},{}^*\}$ of $\CX$, where, for $S\in\CXX$ such that $\scal{S}{1_{X^*}}=0$, one has \cite{berstel}
\begin{eqnarray*}
S^*=\sum_{k\ge0}S^k.
\end{eqnarray*}}
on non commutative variables, ${\mathbb C}^{\mathrm{rat}}\langle\langle X\rangle\rangle$, as follows
\begin{eqnarray*}
S&=&\sum_{n\ge0}\scal{S}{x_0^n}x_0^n+\sum_{k\ge1}\sum_{w\in(x_0^*x_1)^kx_0^*}\scal{S}{w}\;w,\\
\Li_{S}(z)&=&\sum_{n\ge0}\scal{S}{x_0^n}\frac{\log^n(z)}{n!}+\sum_{k\ge1}\sum_{w\in(x_0^*x_1)^kx_0^*}\scal{S}{w}\Li_w,
\end{eqnarray*}
the morphism $\Li_{\bullet}$ is no longer injective over ${\mathbb C}^{\mathrm{rat}}\langle\langle X\rangle\rangle$
but $\{\Li_w\}_{w\in X^*}$ are still linearly independant over $\calC$ \cite{cade,words03}. 
\begin{example}\label{example1}
\begin{enumerate}
\item[i.] $1_{\Omega} =\Li_{1_{X^*}}=\Li_{x_1^*-x_0^*\minishuffle x_1^*}$.
\item[ii.] $\lambda =\Li_{(x_0+x_1)^*}=\Li_{x_0^*\minishuffle x_1^*}=\Li_{x_1^*-1}$.
\item[iii.] $\calC =\C [\Li_{x_0^*},\Li_{(-x_0)^*},\Li_{x_1^*}]$.
\item[iv.] $\calC\{\Li_w\}_{w\in X^*} =\{\Li_S\vert S\in{\mathbb C}[x_0^*]\shuffle{\mathbb C}[(-x_0)^*]\shuffle{\mathbb C}[x_1^*]\shuffle\CX\}.$
\end{enumerate}
\end{example}

Let us consider also the differential and integration operators, acting on $\calC\{\Li_w\}_{w\in X^*}$ \cite{acta}~:
\begin{eqnarray*}
&&\partial_z=\Frac{d}{dz},\theta_0=z\Frac{d}{dz},\theta_1=(1-z)\Frac{d}{dz},\\
\forall f\in\calC,&&\iota_0(f)=\int_{z_0}^zf(s)\omega_0(s)\quad\mbox{and}\quad\iota_1(f)=\int_0^zf(s)\omega_1(s).
\end{eqnarray*}
Here, the operator $\iota_0$ is well-defined (as in definition \ref{def1} in section \ref{continuity}) then one can check easily \cite{orlando,words03,GHMposter,QED}

\begin{enumerate}
\item The subspace $\calC\{\Li_w\}_{w\in X^*}$ is closed under the action of $\{\theta_0,\theta_1\}$ and $\{\iota_0,\iota_1\}$.

\item The operators $\{\theta_0,\theta_1,\iota_0,\iota_1\}$ satisfy in particular,
\begin{eqnarray*}
\theta_1+\theta_0=\bigl[\theta_1,\theta_0\bigr]=\partial_z&\mbox{and}&\forall k=0,1,\theta_k\iota_k=\mathrm{Id},\\
\left[\theta_0\iota_1,\theta_1\iota_0\right]=0&\mbox{and}&(\theta_0\iota_1)(\theta_1\iota_0)=(\theta_1\iota_0)(\theta_0\iota_1)=\mathrm{Id}.
\end{eqnarray*}

\item $\theta_0\iota_1$ and $\theta_1\iota_0$ are scalar operators within $\calC\{\Li_w\}_{w\in X^*}$,
respectively with eigenvalues $\lambda$ and $1/\lambda$, {\it i.e.}
\begin{eqnarray*}
(\theta_0\iota_1)f=\lambda f&\mbox{and}&(\theta_1\iota_0)f=(1/\lambda)f.
\end{eqnarray*}

\item Let $w=y_{s_1}\ldots y_{s_r}\in Y^*$ (then $\pi_X(w)=x_0^{s_1-1}x_1\ldots x_0^{s_r-1}x_1$)
and $u=y_{t_1}\ldots y_{t_r}\in Y_0^*$. The functions $\Li_w,\Li^-_u$ satisfy
\begin{eqnarray*}
\Li_w=({\iota_0^{s_1-1}\iota_1\ldots\iota_0^{s_r-1}\iota_1})1_{\Omega},
&&
\Li^-_u=({\theta_0^{t_1+1}\iota_1\ldots\theta_0^{t_r+1}\iota_1})1_{\Omega},\\
\iota_0\Li_{\pi_X(w)}=\Li_{x_0\pi_X(w)},&&\iota_1\Li_{w}=\Li_{x_1\pi_X(w)},\\
\theta_0\Li_{x_0\pi_X(w)}=\Li_{\pi_X(w)},&&\theta_1\Li_{x_1\pi_X(w)}=\Li_{\pi_X(w)},\\
\theta_0\Li_{x_1\pi_X(w)}={\lambda}\Li_{\pi_X(w)},&&\theta_0\Li_{x_1\pi_X(w)}=\Li_{\pi_X(w)}/{\lambda}.
\end{eqnarray*}
\end{enumerate}

Here, we explain the whole project of extension of $\Li_\bullet$, study different aspects of it, in particular what is desired of $\iota_i,i=0,1$.
The interesting problem in here is that what we do expect of $\iota_i,i=0,1$. In fact, the answers are
\begin{itemize}
\item it is a section of $\theta_i,i=0,1$ ({\it i.e.} takes primitives for the corresponding differential operators).
\item it extends $\iota_i,i=0,1$ (defined on $\C\{\Li_w\}_{w\in X^*}$, and very surprisingly, although not coming directly from Chen calculus, they provide a group-like generating series)
\end{itemize}
We will use this construction to extend $\Li_\bullet$ to $\calC\{\Li_w\}_{w\in X^*}$ 
and, after that, we extend it to a much larger rational algebra.

\section{Background}
\subsection{Standard topology on $\calH(\Omega)$}\label{standard}
The algebra $\calH(\Omega)$ is that of analytic functions defined over $\Omega$. It is endowed with the topology of {\it compact convergence} whose seminorms are indexed by compact subsets of $\Omega$, and defined by
\begin{eqnarray*}
p_K(f)=||f||_K=\sup_{s\in K} |f(s)|.
\end{eqnarray*}
Of course, 
\begin{eqnarray*}
p_{K_1\cup K_2} = \sup (p_{K_1},p_{K_2}),
\end{eqnarray*}
and therefore the same topology is defined by extracting a {\it fondamental subset of seminorms}, here it can be choosen denumerable. As $\calH(\Omega)$ is complete with this topology it is a Frechet space and even, as 
\begin{eqnarray*}
p_K(fg)\leq p_K(f)p_K(g),
\end{eqnarray*}
 it is a Frechet algebra (even more, as $p_K(1_\Omega)=1$ a Frechet algebra with unit).

With the standard topology above, an operator $\phi\in \mathrm{End}(\calH(\Omega))$ is continuous iff 
(with $K_i$ compacts of $\Omega$)
\begin{eqnarray*}
(\forall K_2)(\exists K_1)(\exists M_{12}>0)(\forall f\in \calH(\Omega))
(||\phi(f)||_{K_2}\leq M_{12}||f||_{K_1}).
\end{eqnarray*}
\subsection{Study of continuity of the sections $\theta_i$ and $\iota_i$}\label{continuity}
Then, $\calC\{\Li_w\}_{w\in X^*}$ (and $\calH(\Omega)$) being closed under the operators $\theta_i;\ i=0,1$. We will first build sections of them, then see that they are continuous and, propose (discontinuous) sections more adapted to renormalisation and the computation of associators.      

For $z_0\in \Omega$, let us define $\iota_i^{z_0}\in \mathrm{End}(\calH(\Omega))$ by 
\begin{eqnarray*}
\iota_0^{z_0}(f)=\int_{z_0}^z f(s)\omega_0(s),&&\iota_1^{z_0}(f)=\int_{z_0}^z f(s)\omega_1(s).
\end{eqnarray*}
It is easy to check that $\theta_i\iota_i^{z_0}=Id_{\calH(\Omega)}$ and that they are continuous on $\calH(\Omega)$ for the topology of compact convergence because for all $K\subset_{compact}\Omega$, we have
\begin{eqnarray*}
|p_K(\iota_i^{z_0}(f)| \leq  p_K(f)\Big[\sup_{z\in K}|\int_{z_0}^z\omega_i(s)|\Big],
\end{eqnarray*}
and, in view or paragraph (\ref{standard}), this is sufficient to prove continuity. Hence the operators $\iota_i^{z_0}$ are also well defined on 
$\calC\{\Li_w\}_{w\in X^*}$ and it is easy to check that 
\begin{eqnarray*}
\iota_i^{z_0}(\calC\{\Li_w\}_{w\in X^*})\subset \calC\{\Li_w\}_{w\in X^*}.
\end{eqnarray*}
Due to the decomposition of $\calH(\Omega)$ into a direct sum of closed subspaces
\begin{eqnarray*}
\calH(\Omega)=\calH_{z_0\mapsto 0}(\Omega)\oplus \C 1_{\Omega},
\end{eqnarray*}
it is not hard to see that the graphs of $\theta_i$ are closed, thus, the  $\theta_i$ are also continuous. 

Much more interesting (and adapted to the explicit computation of associators) we have the operator $\iota_i$ (without superscripts), mentioned in the introduction and (rigorously) defined by means of a $\C$-basis of 
\begin{eqnarray*}
\calC\{\Li_w\}_{w\in X^*}=\calC\otimes_\C \C\{\Li_w\}_{w\in X^*}.
\end{eqnarray*}
As $\C\{\Li_w\}_{w\in X^*}\cong \C[\Lyn(X)]$, one can partition the alphabet of this  polynomial algebra in $(\Lyn(X)\cap X^*x_1)\sqcup \{x_0\}$ and then get the decomposition 
\begin{equation*}
\calC\{\Li_w\}_{w\in X^*}=
\calC\otimes_\C \C\{\Li_w\}_{w\in X^*x_1}\otimes_\C \C\{\Li_w\}_{w\in x_0^*}.
\end{equation*}
In fact, we have an algorithm to convert $\Li_{ux_1x_0^n}$ as
\begin{eqnarray*}
\Li_{ux_1x_0^n}(z)=\Sum_{m\leq n}P_m\log^m(z) =\Sum_{m\leq n\atop w\in X^*x_1}\scal{P_m}{w}\Li_w(z)\log^m(z).
\end{eqnarray*}
due to the identity \cite{IMACS}
\begin{eqnarray*}
ux_1x_0^n=ux_1\shuffle x_0^n-\Sum_{k=1}^n(u\shuffle x_0^k)x_1x_0^{n-k}.
\end{eqnarray*}
 This means that 
\begin{eqnarray*}
\calB&=& \Big(z^k \Li_{ux_1}(z)\Li_{x_0^n}(z)\Big)_{(k,n,u)\in \Z\times \N\times X^*}\cr
&\sqcup&\Big(\frac{1}{(1-z)^l}\Li_{ux_1}(z)\Li_{x_0^n}(z)\Big)_{(l,n,u)\in \N^2_+\times X^*}\cr
&\sqcup& \Big(z^k \Li_{x_0^n}(z)\Big)_{(k,n)\in \Z\times \N_+} \cr
&\sqcup& \Big(\frac{1}{(1-z)^l}\Li_{x_0^n}(z)\Big)_{(k,l,n)\in \N^2_+}
\end{eqnarray*}
is a $\C$-basis of $\calC\{\Li_w\}_{w\in X^*}$.\\
With this basis, we can define the operator $\iota_0$ as follows 
\begin{definition}\label{def1}
Define the index map $\ind : \calB\rightarrow \Z$ by
\begin{eqnarray*}
\ind(\frac{z^k}{(1-z)^l}\Li_{x_0^n}(z))&=&k,\\
\ind (\frac{z^k}{(1-z)^l}\Li_{ux_1}(z)\log^n(z))&=&k+|ux_1|.
\end{eqnarray*}
Now $\iota_0$ is computed by:
\begin{enumerate}
\item $\iota_0(b)=\int_{0}^z b(s)\omega_0(s)$ if $\ind(b)\geq 1$.
\item $\iota_0(b)=\int_{1}^z b(s)\omega_0(s)$ if $\ind(b)\leq 0$.
\end{enumerate}
\end{definition}

To show discontinuity of $\iota_0$ with a direct example, the technique consists in exhibiting 2 sequences $f_n,g_n\in \C\{\Li_w\}_{w\in X^*}$ converging to the same limit but such that 
\begin{eqnarray*}
\lim\iota_0(f_n)\not=\lim\iota_0(g_n).
\end{eqnarray*}
Here we choose the function $z$ which can be approached by two limits. For continuity, we should have ``equality of the limits of the image-sequences'' which is not the case.    
\begin{eqnarray*}
z&=& \sum_{n\geq 0}\frac{\log^n(z)}{n!},\\
z&=&\sum_{n\geq 1}\frac{(-1)^{n+1}}{n!}\log^n(\frac{1}{1-z}).
\end{eqnarray*}
Let then 
\begin{eqnarray*}
f_n=\sum_{0\leq m\leq n}\frac{\log^m(z)}{m!}&\mbox{and}&
g_n=\sum_{1\leq m\leq n}\frac{(-1)^{m+1}}{m!}\log^m(\frac{1}{1-z}),
\end{eqnarray*}
we have $f_n,g_n\in \C\{\Li_w\}_{w\in X^*}$ and $\iota_0(f_n)=f_{n+1}-1$.
Hence, one has $\lim(\iota_0(f_n))=z-1$. On the other hand
\begin{eqnarray*}
\lim\iota_0(g_n)=\lim\int_0^zg_n(s)\omega_0(s)=\int_0^z\lim g_n(s)\omega_0(s)=\int_0^zs\omega_0(s)=z.
\end{eqnarray*}
The exchange of the integral with the limit above comes from the fact that the operator 
\begin{eqnarray*}
\phi\mapsto \int_{0}^z\phi(s)\omega_0(s), 
\end{eqnarray*}
is continuous on the space $\calH_0(\Omega\cup B(0,1))$
of analytic functions $f\in \calH(\Omega\cup B(0,1))$ such that $f(0)=0$ 
($B(0,1)$ is the open ball of center $0$ and radius $1$). 
\section{Algebraic extension of $\mathbf{\Li_{\bullet}}$ to \\
$\mathbf{(\ncs{\C^\mathrm{rat}}{X},\shuffle,1_{X^*})[x_0^*,(-x_0)^*,x_1^*]}$}\label{shuff_ext}

We will use several times the following lemma which is characteristic-free.  

\begin{lemma}\label{diff_lemma}
Let $(\calA,d)$ be a commutative differential ring without zero divisor, and $R=\ker(d)$ be its subring of constants.
Let $z\in \calA$ such that $d(z)=1$ and $S=\{e_\alpha\}_{\alpha\in I}$ be a set of eigenfunctions of $d$ all different ($I\subset R$)
{\it i.e}. 
\begin{itemize}
\item[i.] $e_\alpha\not=0$.
\item[ii.] $d(e_\alpha)=\alpha e_\alpha\ ;\ \alpha\in I$.
\end{itemize} 
Then the family $(e_\alpha)_{\alpha\in I}$ is linearly free over $R[z]$\footnote{Here $R[z]$ 
is understood as ring adjunction i.e. the smallest subring generated by $R\cup \{z\}$.}.    
\end{lemma}
\begin{proof}
If there is no non-trivial $R[z]$-linear relation, we are done. Otherwise let us consider  relations 
\begin{equation}\label{lin_rel0}
\sum_{j=1}^{N}P_j(z)e_{\alpha_j}=0,
\end{equation}
with $P_j\in R[t]_{pol}\setminus \{0\}$\footnote{Here 
$R[t]_{pol}$ means the formal univariate polynomial ring (the subscript is here to avoid confusion).} for all $j$ (packed linear relations). We choose one minimal w.r.t. the triplet 
\begin{equation}\label{triplet}
[N,\deg(P_N),\sum_{j<N}\deg(P_j)],
\end{equation}
lexicographically ordered from left to right\footnote{i.e. consider the ones with $N$ minimal and among these,
we choose one with $\deg(P_N)$ minimal and among these we choose one with $\sum_{j<N}\deg(P_j)$ minimal.}.

Remarking that $d(P(z))=P'(z)$ (because $d(z)=1$), we apply the operator $d-\alpha_N Id$ to both sides of (\ref{lin_rel0}) and get
\begin{equation}\label{lin_rel01}
\sum_{j=1}^{N}\Big(P'_j(z)+(\alpha_j-\alpha_N)P_j(z)\Big)e_{\alpha_j}=0.
\end{equation}
Minimality of (\ref{lin_rel0}) implies that (\ref{lin_rel01}) is trivial i.e. 
\begin{equation}\label{coeffs_rel01}
P'_N(z)=0\ ;\ (\forall j=1..N-1)(P'_j(z)+(\alpha_j-\alpha_N)P_j(z)=0).
\end{equation} 
Now relation (\ref{lin_rel0}) implies 
\begin{equation}\label{lin_rel4}
\prod_{1\leq j\leq n-1}(\alpha_N-\alpha_j)\Big(\sum_{j=1}^{N}P_j(z)e_{\alpha_j}\Big)=0,
\end{equation}
which, because $\calA$ has no zero divisor, is packed and has the same associated triplet 
(\ref{triplet}) as (\ref{lin_rel0}). From (\ref{coeffs_rel01}), we see that for all $k<N$ 
\begin{eqnarray*}
\prod_{1\leq j\leq n-1}(\alpha_N-\alpha_j)P_k(z)=
\prod_{1\leq j\leq n-1\atop j\not=k}(\alpha_N-\alpha_j)P'_k(z),
\end{eqnarray*}
so, if $N\geq 2$, we would get a relation of lower triplet (\ref{triplet}). This being impossible, we get $N=1$ and (\ref{lin_rel0}) boils down to $P_N(z)e_N=0$
which, as $\calA$ has no zero divisor, implies $P_N(z)=0$, contradiction.

Then the $(e_\alpha)_{\alpha\in I}$ are $R[z]$-linearly independent.     
\end{proof}

\begin{remark}\label{char_zero}
If $\calA$ is a $\Q$-algebra or only of characteristic zero ({\it i.e.}, $n1_\calA=0\Rightarrow n=0$), 
then $d(z)=1$ implies that $z$ is transcendent over $R$.
\end{remark}
First of all, let us prove 
\begin{lemma}\label{int-dom} Let $k$ be a field of characteristic zero and $Z$ an alphabet. 
Then $(\ncs{k}{Z},\shuffle,1_{Z^*})$ is a $k$-algebra without zero divisor.
\end{lemma}

\begin{proof}
Let $B=(b_i)_{i\in I}$ be an ordered basis of $\Lie_k\langle Z\rangle$ and 
$(\frac{B^{\alpha}}{\alpha !})_{\alpha\in \N^{(I)}}$ the divided corresponding PBW basis. One has 
\begin{eqnarray*}
\Delta_{\shuffle}(\frac{B^{\alpha}}{\alpha !})=
\sum_{\alpha=\alpha_1+\alpha_2}
\frac{B^{\alpha_1}}{\alpha_1 !}\otimes\frac{B^{\alpha_2}}{\alpha_2 !}\ .
\end{eqnarray*}   
Hence, if $S,T\in (\ncs{k}{Z},\shuffle,1_{Z^*})$, considering
\begin{eqnarray*}
\scal{S\shuffle T}{\frac{B^{\alpha}}{\alpha !}}=\scal{S\otimes T}{\Delta_{\shuffle}(\frac{B^{\alpha}}{\alpha !})}=\sum_{\alpha=\alpha_1+\alpha_2}
\scal{S}{\frac{B^{\alpha_1}}{\alpha_1 !}}\scal{T}{\frac{B^{\alpha_2}}{\alpha_2 !}}\ ,
\end{eqnarray*}
we see that $(\ncs{k}{Z},\shuffle,1_{Z^*})\simeq (k[[I]],.,1)$ (commutative algebra of formal series) which has no zero divisor). 
\end{proof}

\begin{lemma}\label{alg_ind}
Let $\mathcal{A}$ be a $\mathbb{Q}$-algebra (associative, unital, commutative) and $z$ an indeterminate, then $e^z\in \mathcal{A}[[z]]$ is transcendent over $\mathcal{A}[z]$. 
\end{lemma} 
\begin{proof}
It is straightforward consequence of Remark (\ref{char_zero}). Note that this can be rephrased as $[z,e^z]$ are algebraically independant over $\mathcal{A}$.
\end{proof}

\begin{proposition}
Let $Z=\{z_n\}_{n\in \N}$ be an alphabet, then $[e^{z_0},e^{z_1}]$ is algebraically independent on $\C[Z]$ within $\C[[Z]]$.   
\end{proposition}
\begin{proof} Using lemma \ref{alg_ind}, one can prove by recurrence that
\begin{eqnarray*}
[e^{z_0},e^{z_1},\cdots ,e^{z_k}, z_0,z_1,\cdots ,z_k],
\end{eqnarray*}
are algebraically independent.  This implies that $Z\sqcup \{e^z\}_{z\in Z}$ is an algebraically independent set and, 
by restriction $Z\sqcup \{e^{z_0},e^{z_1}\}$ whence the proposition.  
\end{proof}

\begin{corollary}
\begin{itemize}
\item[i.] The family $\{x_0^*,x_1^*\}$ is algebraically independent over $(\ncp{\C}{X},\shuffle,1_{X^*})$
within $(\ncs{\C}{X}^\mathrm{rat},\shuffle,1_{X^*})$.
\item[ii.] $(\ncp{\C}{X},\shuffle,1_{X^*})[x_0^*,x_1^*,(-x_0)^*]$ is a free module over $\ncp{\C}{X}$, the family 
$\{(x_0^*)^{\shuffle k}\shuffle (x_1^*)^{\shuffle l}\}_{(k,l)\in \Z\times \N}$ is a $\ncp{\C}{X}$-basis of it.
\item[iii.] As a consequence,  $\{w\shuffle (x_0^*)^{\shuffle k}\shuffle (x_1^*)^{\shuffle l}\}_{w\in X^*\atop (k,l)\in \Z\times \N}$ is a $\C$-basis of it.
\end{itemize}
\end{corollary}

\begin{proof}
Chase denominators and use a theorem by Radford \cite{reutenauer} with $Z=\Lyn(X)$. 
\end{proof}

\begin{corollary}
There exists a unique morphism $\mu$, from\\ 
$(\ncp{\C}{X},\shuffle,1_{X^*})[x_0^*,(-x_0)^*,x_1^*]$ to $\calH(\Omega)$ defined by 
\begin{itemize}
\item[i.] $\mu(w)= \Li_w$ for any $w \in X^*$,
\item[ii.] $\mu(x_0^*)=z,$
\item[iii.] $\mu((-x_0)^*)= {1}/{z},$
\item[iv.] $\mu(x_1^*)={1}/({1-z}).$
\end{itemize}
\end{corollary}

\begin{definition}
We call $\Li^{(1)}_{\bullet}$ the morphism $\mu$. 
\end{definition}

Remark that the image of $(\ncp{\C}{X},\shuffle,1_{X^*})[x_0^*,(-x_0)^*,x_1^*]$ by $\Li^{(1)}_{\bullet}$ (sect. \ref{shuff_ext}) is exactly $\calC\{\Li_w\}_{w\in X^*}$. And the operator $\iota_0$ defined by means of $\Li_\bullet$ is the same as the one defined by tensor decomposition. We have a diagram as follows
\begin{eqnarray*}
\begin{tikzcd}
(\ncp{\C}{X},\shuffle,1_{X^*})\ar[hook,two heads]{r}{\Li_\bullet}\ar[hook]{d}& 
\C\{\Li_w\}_{w\in X^*} \ar[hook]{d}\\
(\ncp{\C}{X},\shuffle,1_{X^*})[x_0^*,(-x_0)^*,x_1^*]
\ar[two heads]{r}{\Li_\bullet^{(1)}}\ar[hook]{d}& 
\calC\{\Li_w\}_{w\in X^*}\ar[hook]{d} \\
\ncp{\C}{X}\shuffle \ncs{\C^{\mathrm{rat}}}{x_0} \shuffle \ncs{\C^{\mathrm{rat}}}{x_1}\ar{r}{\Li_\bullet^{(2)}} & \calH(\Omega)\\
\ncp{\C}{X}\otimes_\C \ncs{\C^{\mathrm{rat}}}{x_0}\otimes_\C \ncs{\C^{\mathrm{rat}}}{x_1}\ar{u}\ar[dashed]{ur}&
\end{tikzcd}
\end{eqnarray*}
\begin{diagram}
Arrows and spaces obtained in this project (so far).
Among horizontal arrows only $\Li_\bullet$ is into (and is an isomorphism) $\Li_\bullet^{(1)}$ and $\Li_\bullet^{(2)}$ are not into
(for example, the image of the non-zero element $x_0^*\shuffle x_1^*- x_1^*+1$ is zero, see Example \ref{example1}).
The image of $\Li_\bullet^{(2)}$ is presumably  
\begin{eqnarray*}
\mathrm{Im}(SP_\C(X))\{\Li_w\}_{w\in X^*}\simeq 
 \C\{z^\alpha (1-z)^{\beta}\}_{\alpha,\beta\in \C}\otimes_\C \C\{\Li_w\}_{w\in X^*}.
\end{eqnarray*}
\end{diagram} 

\section{Extension to $\mathbf{\ncp{\C}{X}\shuffle\ncs{\C^\mathrm{rat}}{x_0}\shuffle \ncs{\C^\mathrm{rat}}{x_1}}$}\label{Star-of-the-plane-ext}

\subsection{Study of the shuffle algebra $SP_{\C}(X)$} 

Indeed, the set $(a_0x_0+a_1x_1)^*_{a_0,a_1\in \C}$ is a shuffle monoid as 
\begin{eqnarray*}
(a_0x_0+a_1x_1)^*\shuffle (b_0x_0+b_1x_1)^*=((a_0+b_0)x_0+(a_1+b_1)x_1)^*.
\end{eqnarray*}
As there are many relations between these elements (as a monoid it is isomorphic to $\C^2$, hence far from being free), we study here the vector space 
\begin{eqnarray*}
SP_{\C}(X)=\mathrm{span}_{\C}\{(a_0x_0+a_1x_1)^*\}_{a_0,a_1\in \C}.
\end{eqnarray*}

It is a shuffle sub-algebra of $\ncs{(\C)^{\mathrm{rat}}}{x_0}\shuffle \ncs{(\C)^\mathrm{rat}}{x_1}$ which will be called {\it star of the plane}. Note that it is also a shuffle sub-algebra of the algebra $(\ncs{\C^{exch}}{X},\shuffle,1_{X^*})$ of exchangeable series. We can give the 

\begin{definition}\label{exchangeable}
A series is said exchangeable iff whenever two words have the same multidegree (here bidegree) 
then they have the same coefficient within it. Formally for all $u,v\in X^*$
\begin{eqnarray*}
\Big((\forall x\in X)(|u|_x=|v|_x)\Big)\implies \scal{S}{u}=\scal{S}{v}.
\end{eqnarray*}
\end{definition} 

On the other hand, for any $S \in \ncs{\C}{X}$, we can write
\begin{eqnarray*}
S= \sum_{n \geq 0} P_n,
\end{eqnarray*}
 where $P_n \in \C[X]$ such that $\deg P_n = n$ for every $n \geq 0$. Then $S$ is called exchangeable iff $P_n$ are symmetric by permutations of places for every $n \in \N$. If $S$ is written as above then we can write 
\begin{eqnarray*}
P_n = \Sum_{i=0}^n \alpha_{n,i} x_0^i \shuffle x_1^{n-i}.
\end{eqnarray*}
\begin{definition}
Let $S\in\CXX$ (resp. $\CX$) and let $P\in\CX$ (resp. $\CXX$). The left and right {\it residual}
of $S$ by $P$ are respectively the formal power series $P\resg S$ and $S\resd P$ in $\CXX$ defined by
\begin{eqnarray*}
\pol{P\resg S\bv w}=\pol{S\bv wP}&\mbox{(resp.}&\pol{S\resd P\bv w}=\pol{S\bv Pw}).
\end{eqnarray*}
\end{definition}

For any $S\in\CXX$ (resp. $\CX$) and $P,Q\in\CX$ (resp. $\CXX$), we straightforwardly get
\begin{eqnarray*}
P\resg (Q\resg S)&=&PQ\resg S,\\
(S\resd P)\resd Q&=&S\resd PQ,\\
(P\resg S)\resd Q&=&P\resg(S\resd Q).
\end{eqnarray*}
In case $x,y\in X$ and $w\in X^* $, we get\footnote{For any words $u,v\in X^*$, if $u=v$ then $\delta_u^v=1$ else $0$.}
\begin{eqnarray*}
x\resg (wy)=\delta_x^yw &\mbox{and}& xw\resd y=\delta_x^yw.
\end{eqnarray*}

\begin{theorem}
Le $\delta\in\mathfrak{Der}(\CX,\shuffle,1_{X^*})$. Moreover, we suppose  that $\delta$ is locally nilpotent\footnote{
$\phi\in \mathrm{End}(V)$ is said to be locally nilpotent iff, for any $v\in V$, there exists $N\in \N$ s.t. $\phi^N(v)=0$.}.
Then the family $(t\delta)^n/{n!}$ is summable and its sum, denoted $\exp(t\delta)$, is is a one-parameter group of automorphisms of $(\CX,\shuffle,1_{X^*})$.
\end{theorem}

\begin{theorem}
Let $L$ be a Lie series\footnote{{\it i.e.} $\Delta_{\minishuffle}(L)=L\hat\otimes1+1\hat\otimes L$ \cite{reutenauer}.}.
Let $\delta^r_L$ and $\delta^l_L$ be defined respectively by 
\begin{eqnarray*}
\delta^r_L(P):=P\rg L& \mbox{and}&\delta^l_L(P):=L\rd P.
\end{eqnarray*}
Then $\delta^r_L$ and $\delta^l_L$ are locally nilpotent derivations of $(\CX,\shuffle,1_{X^*})$.

Therefore, $\exp(t\delta^r_L)$ and $\exp(t\delta^l_L)$ are one-parameter groups of $Aut(\CX,\shuffle,1_{X^*})$
and 
\begin{eqnarray*}
\exp(t{\delta^r_L})P=P\rg\exp(tL)&\mbox{and}&\exp(t{\delta^l_L})P=\exp(tL)\rd P.
\end{eqnarray*}
\end{theorem} 

It is not hard to see that the algebra $\ncp{\C}{X}\shuffle\ncs{\C^\mathrm{rat}}{x_0}\shuffle \ncs{\C^\mathrm{rat}}{x_1}$
is closed by the shuffle derivations\footnote{These operators are, in fact,
the shifts of Harmonic Analysis and therefore defined as adjoints of multiplication, {\it i.e.}
\begin{eqnarray*}
\forall S\in \ncs{\C}{x},&&\scal{\delta^l_x (S)}{w}=\scal{S}{xw}.
\end{eqnarray*}
} $\delta^l_{x_0}, \delta^l_{x_1}$.
In particular, on it, these derivations commute\footnote{Thus, in this case, the operator $\delta^l_x$ has the 
same meaning as the operator $S\rightarrow S_x$ in \cite{SSchu}, 
$x^{-1}$ in the Theory of Languages and $\circ$ in \cite{berstel,reutenauer}.} with $\delta^r_{x_0}$ and $\delta^r_{x_1}$, respectively, {\it i.e.}, for any $x \in X$ and $S \in \ncp{\C}{X}\shuffle\ncs{\C^\mathrm{rat}}{x_0}\shuffle \ncs{\C^\mathrm{rat}}{x_1}$, one has
\begin{eqnarray*}
\delta^l_x (S) = \delta^r_x(S).
\end{eqnarray*}
 Moreover, one has 
\begin{eqnarray*}
(\alpha \delta^l_{x_0} +\beta \delta^l_{x_1})[(a_0x_0+a_1x_1)^*]=
(\alpha a_0+\beta a_1)[(a_0x_0+a_1x_1)^*],
\end{eqnarray*}  
from this we get that the family $\{(a_0x_0+a_1x_1)^*\}_{a_0,a_1\in \C}$ is linearly free over $\C$ 
\begin{eqnarray*}
SP_\C(X)
=\bigoplus_{(a_0,a_1)\in \C}\C\{(a_0x_0+a_1x_1)^*\}.
\end{eqnarray*}
We can get an arrow of $\Li_\bullet^{(2)}$ type 
$(SP_\C(X),\shuffle,1_{X^*})\longrightarrow \calH(\Omega)$ 
by sending
\begin{eqnarray*}
(a_0x_0+a_1x_1)^*=(a_0x_0)^*\shuffle (a_1x_1)^* &\longmapsto&z^{a_0}(1-z)^{-a_1}.
\end{eqnarray*} 
In particular, for any $n \in \N_+$, one has  
\begin{eqnarray*}
\Li_{ \underbrace{ 0,\ldots,0}_{\quad\mbox{$n$ times}}}^-(z) =\Li^{(2)}_{(nx_0 +n x_1)^*}(z).
\end{eqnarray*}
This arrow is a morphism for the shuffle product.  
 
\subsection{Study of the algebra $\mathbf{\ncp{\C}{X}\shuffle\ncs{\C^\mathrm{rat}}{x_0}\shuffle \ncs{\C^\mathrm{rat}}{x_1}}$}

We will start by the study of the one-letter shuffle algebra, {\it i.e.}
$(\ncs{\C^\mathrm{rat}}{x},\shuffle,1_{x^*})$ and use two times Lemma \ref{diff_lemma} above.

Let us now consider $\calA=\ncs{\C^\mathrm{rat}}{x}; e_\alpha=(\alpha x)^*, \alpha\in \C$
endowed with $d= \delta^l_x$ (which is a derivation for the shuffle) and $z=x$. We are in the conditions of Lemma \ref{diff_lemma} and then $((\alpha x)^*)_{\alpha\in \C}$ is $\C[x]$-linearly free which amounts to say that
\begin{eqnarray*}
B_0 = (x^{\shuffle k}\shuffle(\alpha x)^*)_{k\in \N,\alpha\in\C},
\end{eqnarray*} 
is $\C$-linearly free in $\ncs{\C^\mathrm{rat}}{x}$. 

To see that it is a basis, it suffices to prove that $B_0$ is (linearly) generating. Consider that 
\begin{eqnarray*}
\ncs{\C^\mathrm{rat}}{x}=\{P/Q\}_{P,Q^\in\C[x],Q(0)\not=0},
\end{eqnarray*} 
then, as $\C$ is algebraically closed, we have a basis 
\begin{eqnarray*}
B_1\cup B_2=\{x^{ k}\}_{k\geq 0}\cup\{((\alpha x)^*)^{ l}\}_{\alpha\in \C^*,l\geq 1},
\end{eqnarray*}
and it suffices to see that we can generate $B_2$ by elements of $B_0$, which s a consequence of the two identities
\begin{eqnarray*}
 x\shuffle((\alpha x)^*)^{n}&=&\sum_{j=1}^{n+1} \alpha(n,j)((\alpha x)^*)^{\shuffle j}\quad\mbox{with }\alpha(n,n+1)\not=0,\\ 
 x^{k}\shuffle(\alpha x)^* &=& \frac{1}{k!} ( x^{\shuffle k}\shuffle(\alpha x)^*).
\end{eqnarray*} 

Now, we use again Lemma \ref{diff_lemma} with 
\begin{eqnarray*}
\calA=\ncs{\C^\mathrm{rat}}{x_0}\shuffle \ncs{\C^\mathrm{rat}}{x_1}\subset \ncs{\C}{x_0,x_1},
\end{eqnarray*}
hence without zero divisor (see Lemma \ref{int-dom}), endowed with $d= \delta^l_{x_1}$ then 
$(x_0^{\shuffle k}\shuffle(\alpha x_0)^*)_{k\in \N;\atop\alpha\in \C}$,
is linearly free over $R=\ncs{\C^\mathrm{rat}}{x_0}$.  It is easily seen, using a decomposition like 
\begin{eqnarray*}
S=\sum_{p\geq 0,q\geq 0} \scal{S}{x_0^{\shuffle p}\shuffle x_1^{\shuffle q}} x_0^{\shuffle p}\shuffle x_1^{\shuffle q},
\end{eqnarray*}
that $\ncs{\C^\mathrm{rat}}{x_0}=\ker(d)$ and one obtains then that the arrow 
\begin{eqnarray*}
\ncs{\C^\mathrm{rat}}{x_0}\otimes_\C \ncs{\C^\mathrm{rat}}{x_1}\rightarrow 
\ncs{\C^\mathrm{rat}}{x_0}\shuffle \ncs{\C^\mathrm{rat}}{x_1}\subset \ncs{\C^\mathrm{rat}}{x_0,x_1}
\end{eqnarray*}
is an isomorphism. Hence,
$(x_0^{\shuffle k_0}\shuffle(\alpha_0 x_0)^*\shuffle x_1^{\shuffle k_1}\shuffle(\alpha_1 x_1)^*)_{k_i\in \N;\atop \alpha_i\in \C}$
is a $\C$-basis of $\calA=\ncs{\C^\mathrm{rat}}{x_0}\shuffle \ncs{\C^\mathrm{rat}}{x_1}$.
In order to extend $\Li_\bullet$ to $\calA$ we send 
\begin{eqnarray*}
T(k_0,k_1,\alpha_0,\alpha_1)=x_0^{\shuffle k_0}\shuffle(\alpha_0 x_0)^*\shuffle x_1^{\shuffle k_1}\shuffle(\alpha_1 x_1)^*,
\end{eqnarray*}
to $\log^{k_0}(z)z^{\alpha_0}\log^{k_1}({1}/({1-z}))({1}/({1-z}))^{\alpha_1}$,
and see that the constructed arrow follows multiplication given by
\begin{eqnarray*}
T(j_0,j_1,\alpha_0,\alpha_1)T(k_0,k_1,\beta_0,\beta_1)=T(j_0+k_0,j_1+k_1,\alpha_0+\beta_0,\alpha_1+\beta_1).
\end{eqnarray*}
Using, once more, Lemma \ref{diff_lemma}, one gets
\begin{proposition}
The family $\{(\alpha_0 x_0)^*\shuffle(\alpha_1 x_1)^*\}_{\alpha_i\in \C}$ is a\\ $(\ncp{\C}{X},\shuffle,1)$-basis of 
$\ncp{\C}{X}\shuffle\ncs{\C^\mathrm{rat}}{x_0}\shuffle \ncs{\C^\mathrm{rat}}{x_1},$
then we have a $\C$-basis $\{w\shuffle(\alpha_0 x_0)^*\shuffle(\alpha_1 x_1)^*\}_{\alpha_i\in \C\atop w\in X^*}$ of 
\begin{eqnarray*}
\ncp{\C}{X}\shuffle\ncs{\C^\mathrm{rat}}{x_0}\shuffle \ncs{\C^\mathrm{rat}}{x_1}
&=&\ncp{\C}{X}[\ncs{\C^\mathrm{rat}}{x_0}\shuffle \ncs{\C^\mathrm{rat}}{x_1}]\\
&=&\ncp{\C}{X}\shuffle SP_{\C}(X).
\end{eqnarray*}
\end{proposition}

\begin{proof}
We will use a multi-parameter consequence of Lemma \ref{diff_lemma}.

\begin{lemma}
Let $Z$ be an alphabet, and $k$ a field of characteristic zero. Then, the family $\{e^{\alpha  z}\}_{z\in Z\atop\alpha\in k}\subset k[[Z]]$
 is linearly independent over $k[Z]$. 
\end{lemma}

This proves that, in the shuffle algebra the elements
\begin{eqnarray*}
\{(a_0x_0)^*\shuffle (a_1x_1)^*\}_{a_0,a_1\in \C^2}
\end{eqnarray*}
are linearly independent over $\ncp{\C}{X}\simeq \C[\Lyn(X)]$ within\\ $(\ncs{\C}{X},\shuffle,1_{X^*})$.  
\end{proof}  

Now $\Li_\bullet^{(2)}$ is well-defined and this morphism is not into from 
\begin{eqnarray*}
\ncp{\C}{X}\shuffle\ncs{\C^\mathrm{rat}}{x_0}\shuffle \ncs{\C^\mathrm{rat}}{x_1}&=&\cr
\ncp{\C}{X}[\ncs{\C^\mathrm{rat}}{x_0}\shuffle \ncs{\C^\mathrm{rat}}{x_1}]&=&\ncp{\C}{X}\shuffle SP_\C(X),
\end{eqnarray*} 
to $\mathrm{Im}(\Li^{(2)}_\bullet)$.

\begin{proposition}\label{GH_0}
Let $\Li_\bullet^{(1)} : \ncp{\C}{X}[x_0^*, x_1^*,(-x_0)^*]\to \calH(\Omega)$ then
\begin{itemize}
\item[i.] $\mathrm{Im}(\Li_{\bullet}^{(1)})=\calC\{\Li_w\}_{w\in X^*} $.
\item[ii.] $\ker(\Li_{\bullet}^{(1)})$ is the ideal generated by $x_0^{*} \shuffle x_1^{*} - x_1^* + 1_{X^*}$.
\end{itemize}
\end{proposition}

\begin{proof}
As $\ncp{\C}{X}[x_0^*,x_1^*,(-x_0)^*]$ admits $\{(x_0^*)^{\shuffle k}\shuffle (x_1^*)^{\shuffle l}\}_{k\in \Z}^{l\in \N}$
as a basis for its structure of $\ncp{\C}{X}$-module, it suffices to remark
\begin{eqnarray*}
\Li_{(x_0^*)^{\shuffle k}\shuffle (x_1^*)^{\shuffle l}}^{(1)}(z)=z^k\times \frac{1}{(1-z)^l}
\end{eqnarray*}
is a generating system of $\calC$.

First of all, we recall the following lemma
\begin{lemma}\label{18012016}
Let $M_1$ and $M_2$ be $K$-modules ($K$ is a unitary ring). Suppose $\phi: M_1 \to M_2$ is a linear mapping.
Let $N \subset \ker(\phi)$ be a submodule. If there is a system of generators in $M_1$, namely $\{g_i\}_{i \in I}$, such that 
\begin{enumerate}
\item For any $i\in I\setminus J$,  $g_i \equiv \Sum_{j \in J \subset I} c_i^j g_j [\mod\,N]$,  ($c_i^j \in K;  \forall j \in J$);
\item $\{\phi(g_j)\}_{j \in J}$ is $K$-free in $M_2$; 
\end{enumerate}
then $N = \ker(\phi)$.
\end{lemma}
\begin{proof}
Suppose $P \in \ker(\phi)$. Then $P \equiv \Sum_{j \in J} p_j g_j [\mod\,N]$ with $\{p_j\}_{j \in J} \subset K$.
Then $0 = \phi(P) =\Sum_{j \in J} p_j \phi( g_j )$.
From the fact that $\{\phi(g_J)\}_{J \in J}$ is $K-$ free on $M_2$, we obtain $p_j = 0$ for any $j \in J$.
This means that $P \in N$. Thus $\ker(\phi) \subset N$. This implies that $N = \ker(\phi)$.
\end{proof}

\begin{figure}[htp]
\begin{center}
\definecolor{qqqqff}{rgb}{0.0,0.0,1.0}
\begin{tikzpicture}[line cap=round,line join=round,>=triangle 45,x=1.0cm,y=1.0cm]
\draw[->,color=black] (-3.0,0.0) -- (3.0,0.0);
\foreach \x in {-2.0,-1.0,1.0,2.0}
\draw[shift={(\x,0)},color=blue] (0pt,2pt) -- (0pt,-2pt);
\draw[->,color=black] (0.0,0.0) -- (0.0,3.0);
\foreach \y in {0.0,1.0,2.0}
\draw[shift={(0,\y)},color=red] (2pt,0pt) -- (-2pt,0pt);
\clip(-3.0,0.0) rectangle (3.0,3.0);
\begin{scriptsize}
\draw[color=black] (2.2,2.2) node {$(w, \textcolor[rgb]{0.0,0.0,1.0}{l},\textcolor[rgb]{1.0,0.0,0.0}{k}) $};
\draw[color=black] (0.0,0.0) node {$\bigcirc$};
\draw[color=black] (-0.0,0.2) node {$(1_{X^*}, \times, \times )$};
\draw[color=red] (-0.2,1.8) node {$k$};
\draw[color=black] (0.0,2.0) node {$\cdot$};
\draw[color=black] (1.0,0.0) node {$\cdot$};
\draw[color=black] (-2.1,2.2) node {$(w, \textcolor[rgb]{0.0,0.0,1.0}{-l},\textcolor[rgb]{1.0,0.0,0.0}{k})$};
\draw[color=blue] (2.0, 0.5) node {$l$};
\draw[color=blue] (-2.0,0.2) node {$-l$};
\draw[red] (1,2)--(2,2);
\draw[red] (1,1)--(2,2);
\draw[color=black] (1,2) node {$\triangleleft$};
\draw[color=black] (1,1) node {$\triangleright$};
\draw[color=black] (0.7,2.2) node {$(w, \textcolor[rgb]{0.0,0.0,1.0}{l-1},\textcolor[rgb]{1.0,0.0,0.0}{k})$};
\draw[color=black] (0.7,0.8) node {$(w, \textcolor[rgb]{0.0,0.0,1.0}{l-1},\textcolor[rgb]{1.0,0.0,0.0}{k - 1})$};
\draw[blue] (-1,2)--(-2,2);
\draw[blue] (-2,1)--(-2,2);
\draw[color=black] (-1,2) node {$\triangleright$};
\draw[color=black] (-2,1) node {$\triangledown$};
\draw[color=black] (-0.8,2.2) node {$(w, \textcolor[rgb]{0.0,0.0,1.0}{-l+1},\textcolor[rgb]{1.0,0.0,0.0}{k})$};
\draw[color=black] (-2.2,1.2) node {$(w, \textcolor[rgb]{0.0,0.0,1.0}{-l},\textcolor[rgb]{1.0,0.0,0.0}{k-1})$};
\end{scriptsize}
\end{tikzpicture}
\end{center}
\caption{Rewriting $\mod {\calJ}$ of  $\{w\shuffle(x_0^*)^{\shuffle l}\shuffle(x_1^*)^{\shuffle k}\}_{k\in\N,l \in\Z\atop w\in X^*}$.}
\label{Picture of basis}
\end{figure}

Let now ${\cal J}$ be the ideal generated by 
$x_0^{*} \shuffle x_1^{*} - x_1^* + 1_{X^*}$. It is easily checked, from the following formulas,  (for $l\geq 1$)\footnote{In figure \ref{Picture of basis}, $(w, l, k)$ codes the element $w\shuffle(x_0^*)^{\shuffle l} \shuffle (x_1^*)^{\shuffle k}$.}
\begin{eqnarray*}
w \shuffle x_0^* \shuffle (x_1^*)^{\shuffle l}&\equiv& w \shuffle
(x_1^*)^{\shuffle l}-w \shuffle(x_1^*)^{\shuffle l-1}\ [{\cal J}],\cr
w \shuffle (-x_0)^*\shuffle (x_1^*)^{\shuffle l}&\equiv& w \shuffle
(-x_0)^*\shuffle (x_1^*)^{\shuffle l-1}+w \shuffle(x_1^*)^{\shuffle l}[{\cal J}],    
\end{eqnarray*} 
that one can rewrite $[\mod\, {\cal J}]$ any monomial 
$w \shuffle(x_0^*)^{\shuffle k}\shuffle (x_1^*)^{\shuffle l}$ as a 
linar combination of such monomials with $kl=0$. Then, applying lemma \ref{18012016} to the map $\phi = \Li^{(1)}_{\bullet}$ and the modules
\begin{eqnarray*}
 M_1 = \ncp{\C}{X}[ x_0^*, x_1^*, (-x_0)^*],&M_2 = \calH(\Omega),&N = {\cal J},
\end{eqnarray*}
the families and indices 
\begin{eqnarray*}
\{g_i\} &=& \{w\shuffle (x_1^*)^{\shuffle n} \shuffle (x_0^*)^{\shuffle m} \}_{(w, n,m) \in I},\\
 I &=& X^*\times \N \times \Z, \\ 
 J &=& (X^*\times \N \times \{0\}) \sqcup (X^*\times \{0\} \times \Z),
\end{eqnarray*}
 we have the second point of proposition \ref{GH_0}.
\end{proof}

Of course, we also have  $(x_0^*\shuffle x_1^*-x_1^*+1_{X^*})\subset\ker(\Li^{(2)}_{\bullet})$, but the converse is conjectural.

\section{Applications on polylogarithms}
Let us consider also the following morphisms ${\Im}$ and ${\Theta}$ of algebras 
$\ncp{\C}{X}\rightarrow \mathrm{End}(\calC\{\Li_w\})$ defined by
\begin{itemize}
\item[i.] ${\Im}(w)=\mathrm{Id}$ and ${\Theta}(w)=\mathrm{Id}$, if $w=1_{X^*}$.

\item[ii.] ${\Im}(w)={\Im}(v){\iota_i}$ and ${\Theta}(w)={\Theta}(v){\theta_i}$, if $w=vx_i,x_i\in X,v\in X^*$.
\end{itemize}

For any $n\ge0$ and $u\in X^*,f,g\in\calC\{\Li_w\}_{w\in X^*}$, one has \cite{GHMposter, QED}
\begin{eqnarray*}
\partial_z^n &=& \Sum_{w\in X^n}\mu\circ(\Theta\otimes\Theta)[\Delta_{\shuffle}(w)],\\
\Theta(u)(fg) &=& \mu\circ(\Theta\otimes\Theta)[\Delta_{\shuffle}(u)]\circ(f\otimes g).
\end{eqnarray*}
By extension to complex coefficients, we obtain
\begin{eqnarray*}
\calH_{\tt conc}&\cong&(\C\langle\Theta(X)\rangle,{\tt conc},\mathrm{Id},\Delta_{\shuffle},\epsilon),\\
\calH_{\shuffle}&\cong&(\C\langle\Im(X)\rangle,\shuffle,\mathrm{Id},\Delta_{\tt conc},\epsilon).
\end{eqnarray*}
Hence,
\newpage
\begin{theorem}[derivations and automorphisms]
\item \quad Let $P,Q\in\CX$  (resp. $\C[x_0^*,(-x_0)^*,x_1^*]\shuffle\CX$), $T\in\LCXX$  (resp. $\LCX$).
Then $\Theta(T)$ is a derivation in $(\C\{\Li_w\}_{w\in X^*},\times,1)$ (resp. $(\calC\{\Li_w\}_{w\in X^*},\times,1_{\Omega})$)
and $\exp(t\Theta(T))$ is then a one-parameter group of automorphisms of 
\begin{eqnarray*}
({\mathbb C}\{\Li_w\}_{w\in X^*},\times,1_{\Omega})&\mbox{(resp.}&(\calC\{\Li_w\}_{w\in X^*},\times,1_{\Omega})).
\end{eqnarray*}
\end{theorem}

\begin{proof}
Because $\Li_{P\shuffle Q}=\Li_P\Li_Q,\Theta(T)\Li_{P\shuffle Q}=\Li_{(P\shuffle Q)\rg T}$ and then
$\Theta(T)(\Li_P\Li_Q)=\Li_{(P\shuffle Q)\rg T}=\Li_{(P\rg T)\shuffle Q+P\shuffle(Q\rg T)} =\Li_{(P\rg T)\shuffle Q}+\Li_{P\shuffle(Q\rg T)}=(\Theta(T)\Li_P)\Li_Q+\Li_P(\Theta(T)\Li_Q).$
\end{proof}

\begin{theorem}[extension of $\Li_{\bullet}$]
\item \quad The following map is surjective
\begin{eqnarray*}
({\mathbb C}[x_0^*]\shuffle{\mathbb C}[(-x_0)^*]\shuffle{\mathbb C}[x_1^*]\shuffle\CX,\shuffle,1_{X^*})
&\rightarrow&(\calC\{\Li_w\}_{w\in X^*},\times,1),\\
T&\mapsto&\Im(T)1_{\Omega}.
\end{eqnarray*}
\end{theorem}

One has, for any $u\in Y^*,$
\begin{eqnarray*}
\Li^-_{y_{s_1}u}=\theta_0^{s_1}(\theta_0\iota_1)\Li^-_u=\theta_0^{s_1}(\lambda\Li^-_u)
=\Sum_{k_1=0}^{s_1}\binom{s_1}{k_1}(\theta_0^{k_1}\lambda)(\theta_0^{s_1-k_1}\Li^-_u).
\end{eqnarray*}
Hence, successively \cite{GHMposter},
\begin{eqnarray*}
\Li^-_{y_{s_1}\ldots y_{s_r}}
=\sum_{k_1=0}^{s_1}\sum_{k_2=0}^{s_1+s_2-k_1}\ldots\sum_{k_r=0}^{(s_1+\ldots+s_r)-\atop(k_1+\ldots+k_{r-1})}
\binom{s_1}{k_1}\binom{s_1+s_2-k_1}{k_2}\ldots\cr
\binom{s_1+\ldots+s_r-k_1-\ldots-k_{r-1}}{k_r}(\theta_0^{k_1}\lambda)(\theta_0^{k_2}\lambda)\ldots(\theta_0^{k_r}\lambda),
\end{eqnarray*}
where
\begin{eqnarray*}
\theta_0^{k_i}\lambda(z)=\left\{
\begin{array}{rcl}
\lambda(z),&\mbox{if}&{k_i}=0,\\
\Frac1{1-z}\Sum_{j=1}^{k_i}S_2({k_i},j)j!{\lambda^j(z)},&\mbox{if}&k_i>0.
\end{array}\right.
\end{eqnarray*}
Hence,
\begin{eqnarray*}
\Li^-_{y_{s_1}\ldots y_{s_r}}=\Li_T=\Im(T)1_{\Omega},
\end{eqnarray*}
where $T$ is the following exchangeable rational series
\begin{eqnarray*}
T&=&\sum_{k_1=0}^{s_1}\sum_{k_2=0}^{s_1+s_2-k_1}\ldots\sum_{k_r=0}^{(s_1+\ldots+s_r)-\atop(k_1+\ldots+k_{r-1})}
\binom{s_1}{k_1}\binom{s_1+s_2-k_1}{k_2}\ldots\cr
&&\binom{s_1+\ldots+s_r-k_1-\ldots-k_{r-1}}{k_r}T_{k_1}\shuffle\ldots\shuffle T_{k_r},\\
T_{k_i}&=&\left\{
\begin{array}{rcl}
(x_0+x_1)^*,&\mbox{if}&{k_i}=0,\\
x_1^*\shuffle\Sum_{j=1}^{k_i}S_2({k_i},j)j!((x_0+x_1)^*)^{\shuffle j},&\mbox{if}&k_i>0.
\end{array}\right.
\end{eqnarray*}
Due to surjectivity of $\Li_{\bullet}$, from ${\mathbb C}[x_0^*]\shuffle{\mathbb C}[(-x_0)^*]
\shuffle{\mathbb C}[x_1^*]\shuffle\CX$ to $\calC\{\Li_w\}_{w\in X^*}$,
one also has
\begin{eqnarray*}
\Li^-_{y_{s_1}\ldots y_{s_r}}=\Li_R=\Im(R)1_{\Omega},
\end{eqnarray*}
where $R$ is the following exchangeable rational series
\begin{eqnarray*}
R&=&\sum_{k_1=0}^{s_1}\sum_{k_2=0}^{s_1+s_2-k_1}\ldots\sum_{k_r=0}^{(s_1+\ldots+s_r)-\atop(k_1+\ldots+k_{r-1})}
\binom{s_1}{k_1}\binom{s_1+s_2-k_1}{k_2}\ldots\cr
&&\binom{s_1+\ldots+s_r-k_1-\ldots-k_{r-1}}{k_r}R_{k_1}\shuffle\ldots\shuffle R_{k_r},\\
R_{k_i}&=&\left\{
\begin{array}{rcl}
x_0^*\shuffle x_1^*,&\mbox{if}&{k_i}=0,\\
x_1^*\shuffle\Sum_{j=1}^{k_i}S_2({k_i},j)j!(x_0^*\shuffle x_1^*)^{\shuffle j},&\mbox{if}&k_i>0,
\end{array}\right.
\end{eqnarray*}
and again  (see Example \ref{example1})
\begin{eqnarray*}
\Li^-_{y_{s_1}\ldots y_{s_r}}=\Li_F=\Im(F)1_{\Omega},
\end{eqnarray*}
where $F$ is the following rational series on $x_1$
\begin{eqnarray*}
F&=&\sum_{k_1=0}^{s_1}\sum_{k_2=0}^{s_1+s_2-k_1}\ldots\sum_{k_r=0}^{(s_1+\ldots+s_r)-\atop(k_1+\ldots+k_{r-1})}
\binom{s_1}{k_1}\binom{s_1+s_2-k_1}{k_2}\ldots\cr
&&\binom{s_1+\ldots+s_r-k_1-\ldots-k_{r-1}}{k_r}F_{k_1}\shuffle\ldots\shuffle F_{k_r},\\
F_{k_i}&=&\left\{
\begin{array}{rcl}
x_1^*-1_{X^*},&\mbox{if}&{k_i}=0,\\
x_1^*\shuffle\Sum_{j=1}^{k_i}S_2({k_i},j)j!(x_1^*-1_{X^*})^{\shuffle j},&\mbox{if}&k_i>0.
\end{array}\right.
\end{eqnarray*}
Since $\Im(x_1^*)1_{\Omega}=1/(1-z)$ then this proves once again that \cite{GHMposter,QED}
\begin{eqnarray*}
\Li^-_{y_{s_1}\ldots y_{s_r}}=\Li_T=\L_R=\Li_F\in\C[1/(1-z)]\subsetneq\calC.
\end{eqnarray*}

One can deduce finally that
\begin{corollary}
\begin{eqnarray*}
\calC\{\Li_{w}\}_{w\in  
X^*}&\supsetneq&\C[1/(1-z)]\{\Li_{w}\}_{w\in X^*}\\
&=&\mathrm{span}_{\C}\left\{{\Sum_{n_1>\ldots>n_r>0} n_1^{s_1}\ldots  
n_r^{s_r}\; z^{n_1}\vert
\atop(s_1,\ldots,s_r)\in\Z^r,r\in\N_+}\right\}.
\end{eqnarray*}
\end{corollary}

\section{Conclusion}
We have studied the structure of the algebra 
\begin{equation*}
\ncp{\C}{X}\shuffle\ncs{\C^\mathrm{rat}}{x_0}\shuffle  
\ncs{\C^\mathrm{rat}}{x_1},
\end{equation*}
 where $X = \{x_0, x_1\}$ is an alphabet. We have also considered the ways for denoting the polylogarithms. By the results on the algebra $\ncp{\C}{X}\shuffle\ncs{\C^\mathrm{rat}}{x_0}\shuffle \ncs{\C^\mathrm{rat}}{x_1}$, we have given an extension of the polylogarithms and have obtained {\it polylogarithmic transseries}
\begin{eqnarray*}
 \C\{z^\alpha (1-z)^\beta\Li_w\}_{w\in X^*\atop\alpha,\beta\in\C}.
\end{eqnarray*}
With this extension, we have constructed several shuffle bases of the algebra of polylogarithms. In the special case of the ``Laurent subalgebra'' 
\begin{eqnarray*}
(\ncp{\C}{X},\shuffle,1_{X^*})[x_0^*,(-x_0^*),x_1^*]\subset  
\ncp{\C}{X}\shuffle\ncs{\C^\mathrm{rat}}{x_0}\shuffle  
\ncs{\C^\mathrm{rat}}{x_1},
\end{eqnarray*}
we have completely characterized the kernel of the polylogarithmic map $\Li_\bullet$, providing a rewriting process which terminates to a normal form.

%
%

\end{document}